\documentclass[12pt]{article}

\usepackage{amssymb,amsthm,amsmath,footnote}
\usepackage{colortbl}

\newcommand{\comment}[1]{}

\newtheorem{theorem}{Theorem}[section]

\newtheorem{lemma}[theorem]{Lemma}

\theoremstyle{definition}
\newtheorem{defn}[theorem]{Definition}

\newtheorem*{thm:main}{Theorem \ref{main}}

\def \Z {\mathbb Z}

\title {Orthogonally Resolvable Matching Designs}

\author{
P.\ Danziger \thanks{Supported by the NSERC Discovery grant program, grant \#RGPIN-2016-04178.}
\\
\small Department of Mathematics\\
\small Ryerson University\\
\small Toronto, ON M5B 2K3\\
\small Canada \\
\and
S.\ Park 
\\
\small Department of Mathematics\\
\small Ryerson University\\
\small Toronto, ON M5B 2K3\\
\small Canada \\
}

\date{}

\begin{document}

\maketitle

\begin{abstract}
An Orthogonally resolvable Matching Design OMD$(n, k)$ is a partition of the edges the complete graph $K_n$ into matchings of size $k$, called blocks, such that the blocks can be resolved in two different ways. Such a design can be represented as a square array whose cells are either empty or contain a matching of size $k$, where every vertex appears exactly once in each row and column. In this paper we show that an OMD$(n.k)$ exists if and only if $n \equiv 0 \pmod{2k}$ except when $k=1$ and $n = 4$ or $6$.
\end{abstract}


\section{Introduction}

We assume that the reader is familiar with the general concepts of graph theory and design theory, and refer them to \cite{Handbook, West}. 
In particular, the {\em Lexicographic product} of a graph $G$ with a graph $H$, denoted $G[H]$, is defined as the graph on vertex set $V(G) \times V(H)$ with $(u_G, u_H) (v_G, v_H) \in E(G[H])$ if $u_G v_G \in E(G)$, or $u_G = v_G$ and $u_H v_H \in E(H)$. 
In the case where $H$ is the empty graph on $w$ points, we denote the lexicographic product $G[H]$ by $G[w]$.
We use $K_n$ to denote the complete graph on $n$ vertices and thus $K_n[w]$ is the complete multipartite graph with $n$ parts of size $w$.
Also, a {\em matching} on $2k$ vertices is denoted $M_k$ and is defined as a set of $k$ disjoint edges.

Given two graphs $G$ and $H$, an $G$-decomposition of $H$ is a partition of the edges of $H$ into graphs isomorphic to $G$, called {\em blocks}. The most studied case is when $H$ is the complete graph, in which case we call a $G$-decomposition of $H$ a {\em $G$-design}. 
A {\em resolution class} of a $G$-decomposition of $H$ is a set of blocks which partitions the point set.
A $G$-decomposition of $H$ is called {\em resolvable} if the set of all blocks can be partitioned into resolution classes. In this case each point appears in the same number of blocks and we call this the replication number and denote it by $r$. A resolvable $G$-decomposition of $H$ is also referred to as a {\em $G$-factorization of $H$} and a single class as a {\em
$G$-factor of $H$}. We note that each factor is a spanning subgraph of $H$.

If a $G$-decomposition of $H$ has two resolutions such that the intersection between any class from one resolution with any class from the other is at most one block, then the decomposition is {\em orthogonally resolvable}, sometimes called {\em doubly resolvable}. An orthogonally resolvable $G$-decomposition of $H$ can be represented by an $r\times r$ array, where each cell is either empty or contains a block of the decomposition. Each row and each column is a resolution class and thus contains each point exactly once.

The simplest case of a $G$-decomposition of $H$ is when $G$ is a single edge $K_2\cong M_1$. A resolvable $M_1$-decomposition is also known as a 1-factorization, which are well studied, see \cite{Wallis}. In particular it is well known that a 1-factorization of $K_n$ exists if and only if $n$ is even and a 1-factorization of $K_{n,n} \cong K_2[n] \cong M_1[n]$ exists for all $n\in \Z^+$. 

The study of orthogonal resolutions of designs has a long history.
An orthogonal 1-factorization is called a {\em Room square}, after T.\ G.\ Room \cite{Room} who studied them in the 1950s. However, the study of Room squares goes back to the original work of Kirkman in 1847 \cite{Kirkman}, where he presents a Room square of order 8.  
The existence of Room squares was finally settled in 1975 by Mullin and Wallis \cite{Mullin}; for a survey on Room squares see \cite{blue book}. 

Orthogonally resolvable $K_3$-Designs are known as {\em Kirkman squares}, and have been well studied, see for example~\cite{ALW, CCV, CLLM, Mathon Vanstone}. 
In particular, Mathon and Vanstone \cite{Mathon Vanstone} showed the non-existence of a Kirkman square of orders $n=9$ and $15$; the existence of Kirkman squares was settled by 
Colbourn, Lamken, Ling and Mills in \cite{CLLM}, with 23 possible exceptions, 11 of which were solved in \cite{ALW}.
Another generalisation of $G$-designs that has been considered is when $G$ is an $n$-cycle, see \cite{OCD}.

Different master graphs have also been considered, for example it is easy to see that an orthogonal 1-factorization of $K_2[n]$ is equivalent to a pair of mutually orthogonal Latin squares, which are well known to exist for all $n\neq 2, 6$.

In this paper we consider orthogonally resolvable $M_k$-decompositions and designs. An orthogonally resolvable $M_k$-decomposition of $G$ is denoted OMD$(G, k)$, and when $G=K_n$ we write OMD$(n,k)$. We can also define these decompositions in terms of the corresponding square $r\times r$ array.

\begin{defn}
An {\em OMD($n$, $k$)} is defined as the decomposition of $K_n$ into two orthogonal resolutions with blocks that are isomorphic to $M_k$. It can be represented as a square array of side $n-1$, such that:
\begin{enumerate}
\item Each cell is either empty or contains a copy of $M_k$, where $M_k$ is the matching on $k$ edges.
\item Each row $R$ and each column $S$ contains each element of $V(K_n)$ exactly once.
\item Every pair $x$, $y$ $\in$ $V(K_n)$ occurs together as an edge of one of the $M_k$ exactly once.
\end{enumerate}
\end{defn}

We note the following obvious necessary condition for the existence of any OMD$(n,k)$, which comes from counting the edges and vertices of $K_n$ and $M_k$.
\begin{theorem}
\label{NC}
An OMD$(n, k)$ exists only if $n \equiv 0 \pmod{2k}$.
\end{theorem}

Room squares are thus OMD$(n,1)$ and we state the result of Mullin and Wallis in the following theorem using the language of OMDs.

\begin{theorem}[\cite{Mullin}]
\label{n1}
An $OMD(n,1)$ exists if and only if $n$ is even and $n\neq 4,6$.
\end{theorem}

In the next section we consider some small cases and then we provide a recursive construction and prove our main theorem, which we state here.
\begin{theorem}
\label{main}
There exists an $OMD(n,k)$ if and only if $n \equiv 0 \pmod{2k}$ except when $k=1$ and $n = 4$ or $6$.
\end{theorem}

\section{Small Cases}

In this section we consider some of the ingredients we will need as well as OMD$(mk,k)$ for small values of $m$. We generally work on the $r\times r$ array, we index the rows and columns by $\Z_r$. We begin by defining some terms that we will find useful.

\begin{defn}
A {\em transversal} of an $OMD(n,k)$ is a set of $r$ cells, which contains exactly one cell from each row and one cell from each column, such that each point appears exactly once in one of the cells.
\end{defn}
We note that in general a transversal will contain empty cells.

\begin{defn}
We say that an $OMD(n,k)$ has a {\em hole} of size $m$ if it contains a square subarray of side $m$ which is empty.
\end{defn}

\begin{lemma}
\label{M1k}
There exists an $OMD(M_1[k],k)$ for all $k \in \Z^+$.
\end{lemma}
\begin{proof}
Note that $M_1[k]\cong K_{k,k}$ and each 1-factor of $K_{k,k}$ is isomorphic to $M_k$. We obtain the design by placing the 1-factors of a 1-factorization of $K_{k,k}$ down the diagonal of the square.
\end{proof}

\begin{lemma}
\label{2k}
There exists an OMD$(2k,k)$ with a transversal and a hole of size $k-1$ for all $k \in \Z^+$.
\end{lemma} 
\begin{proof}
We note that each 1-factor of $K_{2k}$ is isomorphic to $M_k$. We obtain the design by placing the 1-factors of a 1-factorization of $K_{2k}$ down the diagonal of the square. The back diagonal $(2k-2-i,i)$, $0 \leq i < 2k-1$, is a transversal, with $i = k-1$ being the only non-empty cell. Further, the upper right $(k-1)\times (k-1)$ subarray is empty.
\end{proof}

\begin{lemma}
\label{4k}
There exists an OMD$(4k,k)$ for all $k>1$.
\end{lemma} 
\begin{proof}
We form the design on point set $X = \Z_{k} \times \{0, 1\} \times \{0, 1\}$. 
Define the partition of $X$ given by
$A_i = \Z_{k} \times \{0\} \times \{i\}$, and $B_i = \Z_{k} \times \{1\} \times \{i\}$, $i \in \{0, 1\}$.

Now for $i, j \in \{0, 1\}$ the edges between $A_i$ and $B_j$ are isomorphic to $K_{k,k}$.
Let $(A_i B_j)_\ell$ be the $\ell^{\rm th}$ 1-factor of a 1-factorization of the edges between $A_i$ and $B_j$, $i, j \in \{0, 1\}$, $\ell\in \{1,\ldots, k\}$. We note that each 1-factor is isomorphic to an $M_k$. The following $k\times k$ squares exactly cover all edges between $A_i$ and $B_j$, the rows and columns being resolution classes.
\[
A = 
\begin{array}{|c|c|c|c|c|c|} \hline
(A_0 B_0)_1 & (A_1 B_1)_1 & & \cdots & & \\ \hline
	    & (A_0 B_0)_2 & (A_1 B_1)_2 & \cdots & & \\ \hline
\vdots &\vdots & \vdots&\;\; \ddots\;\; & \vdots &  \vdots\\ \hline 
& & & \cdots & (A_0 B_0)_{k} & (A_1 B_1)_{k-1} \\ \hline
(A_1 B_1)_{k} &	& &\cdots & &		   (A_0 B_0)_{k} \\ \hline
\end{array}
\]
and
\[
B = 
\begin{array}{|c|c|c|c|c|c|} \hline
(A_0 B_1)_1 & (A_1 B_0)_1 & & \cdots & & \\ \hline
	    & (A_0 B_1)_2 & (A_1 B_0)_2 & \cdots & & \\ \hline
\vdots &\vdots & \vdots&\;\; \ddots\;\; & \vdots &  \vdots\\ \hline 
& & & \cdots & (A_0 B_1)_{k} & (A_1 B_0)_{k-1} \\ \hline
(A_1 B_0)_{k} &	& &\cdots & &		   (A_0 B_1)_{k} \\ \hline
\end{array}
\]
Let $a_i$, $0 < i < 2k$ be the $i^{\rm th}$ 1-factor of a 1-factorization of $A_0\cup A_1$ and $b_i$, $0 < i < 2k$ be the $i^{\rm th}$ 1-factor of a 1-factorization of $B_0\cup B_1$. Again each of these one factors is an $M_k$, and $a_i \cup b_j$ covers all points.
Now, the $(2k-1)\times (2k-1)$ square $C$ below covers all remaining edges.
Note that $k>1$ ensures that $2k-1 \geq 2$, and so room exists to place the factors as shown.
\[
C = 
\begin{array}{|c|c|c|c|c|c|} \hline
a_1 & b_1 & & \cdots & & \\ \hline
	    & a_2 & b_2 & \cdots & & \\ \hline
\vdots &\vdots & \vdots&\;\; \ddots\;\; & \vdots &  \vdots\\ \hline 
& & & \cdots & a_{2k-2} & b_{2k-2} \\ \hline
b_{2k-1} &	& &\cdots & &		   a_{2k-1} \\ \hline
\end{array}
\]
Now, the array below is the required $(4k-1)\times (4k-1)$ array.
\[
\begin{array}{|c|c|c|} \hline
A & & \\ \hline
 & B & \\ \hline
 & & C \\ \hline
\end{array}
\]
\end{proof}

\begin{lemma}
\label{6k}
There exists an OMD$(6k,k)$ for all $k>1$.
\end{lemma} 
\begin{proof}
The square below gives an orthogonal 1-factorization of $K_{2,2,2}$ with point set $X=\Z_3\times \Z_2$.
\[
\begin{array}{|c|c|c|c|} \hline
0_0 1_0 & 0_1 2_0 & 	    & 1_1 2_1   \\ \hline
0_1 2_1 & 0_0 1_1 & 1_0 2_0 & 	 	\\ \hline
1_1 2_0 & 	  & 0_0 2_1 & 0_1 1_0 \\ \hline
	& 1_0 2_1 & 0_1 1_1 & 0_0 2_0 \\ \hline
\end{array}
\]
We construct our design on $X\times \Z_{k}$. Expand each cell of the square above by $k$ and on the expansion of each non-empty cell containing a block $B$, place a copy of an OMD$(M_1[k],k)$ with point set $B\times \Z_k$. The resulting $4k\times 4k$ square, $A$ say, is an orthogonal 1-factorization of $K_{2k,2k,2k}$. Let $a_i$, $b_i$ and $c_i$, $0 < i < 2k$ be the $i^{\rm th}$ 1-factor of a 1-factorization of each of the parts of size $2k$. Form the $(2k-1)\times (2k-1)$ array $B$ below, note that when $k>1$, $2k-1 \geq 3$.
\[
B = 
\begin{array}{|c|c|c|c|c|c|c|} \hline
a_1 & b_1 & \;\;c_1\;\; & & \cdots & & \\ \hline
	    & a_2 & b_2 & \;\;c_2\;\; &\cdots & & \\ \hline
\vdots &\vdots & \vdots& \vdots & \;\; \ddots\;\; & \vdots &  \vdots\\ \hline 
c_{2k-2} & & & & \cdots & a_{2k-2} & b_{2k-2} \\ \hline
b_{2k-1} & c_{2k-1}	& & & \cdots & & 	   a_{2k-1} \\ \hline
\end{array}
\]
Now 
$
\begin{array}{|c|c|} \hline
A & \\ \hline
 & B \\ \hline
\end{array}
$ is the required $(6k-1)\times (6k-1)$ array.
\end{proof}

\section{Main Theorem}

In this section we give a recursive construction and then use it to prove our main theorem.
\begin{theorem}
\label{recurse}
If there exists an $OMD(n,l)$ with a transversal, an $OMD(M_l[s],k)$ and an $OMD(M_l[K_s], k)$ with a transversal and a hole of size $s-1$, then there exists an $OMD(sn,k)$ with a transversal.
\end{theorem}
\begin{proof}
Begin with an $OMD(n,l)$ with a transversal, on point set $X$, $\cal M$. We will construct the new design on point set $X \times \Z_s$. Expand each cell of the array into an $s \times s$ subarray and add $s-1$ extra rows and columns. Since an $OMD(n,l)$ has side $n-1$, and $s(n-1) + s - 1 = sn-1$, we now have an appropriately sized array on which to place an $OMD(sn, k)$. 

Firstly, consider the expansions of the cells that made up the transversal in the $OMD(n,l)$. On the expansion of each non-empty cell containing a block $B$ of the transversal, place the a copy of the $OMD(M_l[K_s], k)$ with point set $V(B) \times \Z_s$. Place the last $s-1$ rows and columns into the extra rows and columns that were added, in such a way that the hole is placed over the $(s-1)\times(s-1)$ common cells as indicated in the figure below.

\[
\begin{array}{|c|c|c|c|}  \hline
{\cellcolor[gray]{.8}} Hole & \ldots & {\cellcolor[gray]{.8}} & \ldots \\ \hline
\vdots &  \ddots & \vdots &  \ddots \\  \hline
{\cellcolor[gray]{.8}} & \ldots & {\cellcolor[gray]{.8}} OMD(M_l[K_s],k) & \ldots  \\  \hline
\vdots &  \ddots & \vdots & \ddots \\  \hline
\end{array}
\]

Now, on the expansion of every other non-empty cell containing a block $B$, place a copy of an OMD($M_l[s]$, k) with point set $V(B) \times \Z_s$ in such a way that the expansion of each point remains empty. 
It is straightforward to see that the result is the required design.
\end{proof}

We are now ready to prove our main result.
\begin{thm:main}
There exists an $OMD(n,k)$ if and only if $n \equiv 0 \pmod{2k}$ except when $k=1$ and $n = 4$ or $6$.
\end{thm:main}
\begin{proof}
We first note that the condition $n \equiv 0 \pmod{2k}$ is necessary by Theorem~\ref{NC}. 
The case $k=1$ is Theorem~\ref{n1}.
An OMD$(2k,k)$ exists by Lemma~\ref{2k}, an OMD$(4k,k)$, $k>1$, exists by Lemma~\ref{4k} and an OMD$(6k,k)$, $k>1$, exists by Lemma~\ref{6k}.
 
Now assume that $k>1$ and $n\geq 8k$, we apply Theorem~\ref{recurse} with $\ell=1$ and $s=k$. An OMD$(n,1)$ exists by Theorem~\ref{n1} and an OMD$(M_1[k],k)$ exists by Lemma~\ref{M1k}. We note that $M_1[K_k]\cong K_{2k}$ and so an OMD$(M_1[K_k],k)$ with a transversal and a hole of size $k-1$ exists by Lemma~\ref{2k}.
\end{proof}

\end{document}